\documentclass[a4paper,11pt]{article}
\usepackage[utf8]{inputenc}
\usepackage[T1]{fontenc}

\usepackage{amsthm,amsmath}
\usepackage{tikz}
\usepackage{mathrsfs,amssymb,amsfonts} 
\usepackage{enumitem}
\usepackage{fullpage}
\usepackage{hyperref, enumerate}
\usepackage[babel]{microtype}
\usepackage[english]{babel}
\usepackage[capitalise]{cleveref}

\usepackage{thmtools}
\usepackage{mathtools}
\usepackage{amssymb}
\usepackage[nomath]{lmodern}
\usepackage{graphicx}
\usepackage{pgf,tikz,tkz-graph,subcaption}
\usetikzlibrary{arrows,shapes}
\usetikzlibrary{decorations.pathreplacing}
\usepackage{tkz-berge}
\usepackage{enumitem}
\usepackage[normalem]{ulem}
\usepackage{hyperref}
\hypersetup{colorlinks = true, linkcolor = blue, citecolor = blue, urlcolor = blue}

\newcommand*{\abs}[1]{\left\lvert #1\right\rvert}

\allowdisplaybreaks

\usepackage[margin=1in]{geometry}

\newtheorem{defi}{Definition}

\newtheorem{thm}[defi]{Theorem}
\newtheorem{lem}[defi]{Lemma}

\newtheorem{claim}[defi]{Claim}
\newcommand*{\myproofname}{Proof}
\newenvironment{claimproof}[1][\myproofname]{\begin{proof}[#1]}{\end{proof}}

\title{Resolution of two conjectures by Erd\H{o}s and Hall concerning separable numbers }
\author{Stijn Cambie \thanks{Department of Computer Science, KU Leuven Campus Kulak-Kortrijk, 8500 Kortrijk, Belgium. Supported by a postdoctoral fellowship by the Research Foundation Flanders (FWO) with grant number 1225224N. Email: \protect\href{mailto:stijn.cambie@hotmail.com}{\protect\nolinkurl{stijn.cambie@hotmail.com}}}
\and Wouter van Doorn\thanks{\protect\href{wonterman1@hotmail.com}{\protect\nolinkurl{wonterman1@hotmail.com}}}}

\begin{document}
\parindent=0cm
\maketitle

\begin{abstract}
    Erd\H{o}s and Hall defined a pair $(m, n)$ of positive integers to be interlocking, if between any pair of consecutive divisors (both larger than $1$) of $n$ (resp. $m$) there is a divisor of $m$ (resp. $n$). 
    A positive integer is said to be separable if it belongs to an interlocking pair. We prove that the lower density of separable powers of two is positive, as well as the lower density of powers of two which are not separable. Finally, we prove that the number of interlocking pairs whose product is equal to the product of the first primes, is finite. We hereby resolve two conjectures by Erd\H{o}s and Hall.
\end{abstract}

\section{Introduction}\label{sec:intro}
Following the terminology of Erd\H{o}s and Hall~\cite{EH78}, we say that two positive integers $m$ and $n$ \emph{interlock} if between every two divisors of $n$ (both larger than $1$) there exists a divisor of $m$, and between every two divisors of $m$ (both larger than $1$) there exists a divisor of $n$. If this holds, then we call $(m, n)$ an interlocking pair. For example, $(63, 64)$ is an interlocking pair, as we can order their divisors as follows: $$1 = 1 < 2 < 3 < 4 < 7 < 8 < 9 < 16 < 21 < 32 < 63 < 64,$$
and note that the divisors of $63$ and $64$ alternate here.

A positive integer $n$ is defined to be \emph{separable}, if a positive integer $m$ exists such that $m$ and $n$ interlock. In particular, both $63$ and $64$ are separable, which incidentally contradicts the claim in \cite{EH78} that $2^k$ cannot be separable if $k+1 \ge 5$ is prime. The reason that $2^k$ can interlock with an integer $m$, even if $k+1$ is prime, is that it is possible for the number of divisors of $2^k$ and $m$ to differ by $1$, see \cref{lem:tau}.

In any case, denote with $A(x)$ the number of separable $n \le x$. In \cite{EH78}, Erd\H{o}s and Hall stated that they could not settle whether $A(x) = o(x)$ holds, and this still seems like the most important open problem regarding separable numbers. On the other hand, they did manage to prove the lower bound $A(x) > \frac{cx}{\log \log x}$, for some absolute constant $c > 0$ and all large enough $x$. 
As a positive integer seems more likely to be separable if its divisors are spread out, Erd\H{o}s and Hall then conjectured that $n = 2^k$ is separable for almost all $k$. They furthermore conjectured that $m$ and $n$ do not interlock if $mn$ is the product of the first $k$ primes and $k$ is large enough. In this paper we refute the first conjecture (see Theorem~\ref{thm:2^k_separable} and Theorem~\ref{thm:posdens_2^kseparable}), and confirm the second one (see Theorem~\ref{thm:onlyfewproducts}).
A computer verification of our results with Lean can be found at \cite{LeanFile}.

\section{Main results and proofs}\label{sec:results}
Let $\tau(n)$ be the number of divisors of $n$, and let $d_2(n)$ denote the smallest prime divisor of $n$. We then start off with a lemma on interlocking pairs that quickly follows from their definition.

\begin{lem}\label{lem:tau}
    Let $(m, n)$ be an interlocking pair with $d_2(n) < d_2(m)$. If $n < m$, then $\tau(m) = \tau(n)$, and if $n > m$, then $\tau(m) = \tau(n) - 1$.
\end{lem}

\begin{proof}
    Let the divisors of $m$ be $$1 = d_1 < d_2 < \cdots < d_{\tau(m)} = m,$$ and let the divisors of $n$ be $$1 = d'_1 < d'_2 < \cdots < d'_{\tau(n)} = n.$$
    If $n < m,$
    then $$d'_2<d_2<d'_3<d_3< \cdots <d'_{\tau(n)} < d_{\tau(m)},$$ implying $\tau(m)=\tau(n)$.
    And if $n > m,$
    then $$d'_2<d_2<d'_3<d_3< \cdots <d_{\tau(m)}<d'_{\tau(n)},$$ implying $\tau(m)=\tau(n)-1$.
\end{proof}

In \cite{EH78} it is conjectured that $2^k$ is separable for almost all $k$. However, as it turns out, the lower density of $k$ for which $2^k$ is not separable, is positive.

\begin{thm} \label{thm:2^k_separable}
    For all $k > 2$ with $k \equiv 1,2,9,10 \pmod{12}$, we have that $n = 2^k$ is not separable.
\end{thm}

\begin{proof}
    By contradiction; assume that $m$ and $n = 2^k$ interlock, for some $k \equiv 1,2,9,10 \pmod{12}$ with $k > 2$. Write $m=\prod_i p_i^{e_i}$ and recall that we have $\tau(m)=\prod_i (e_i+1).$ We will then deal with $\tau(m)$ odd and even separately.

    If $\tau(m)$ is odd, then $e_i$ must be even for all $i$, so that $m$ is a square. Since $3$ is the only integer between $2$ and $4$, and $5,7$ are the only odd integers between $4$ and $8$, we then see that either $15^2 \mid m$ or $21^2 \mid m$. However, $15^2$ has the two divisors $9$ and $15$ which are both between $8$ and $16$, while $21^2$ has the two divisors $49$ and $63$ which are both between $32$ and $64$; a contradiction in either case.

    If $\tau(m)$ is even, we apply~\cref{lem:tau}. Since $\tau(n) = k+1$ and $d_2(n) < d_2(m)$, we then obtain $\tau(m) \in\{k,k+1\}.$ And because $\tau(m)$ is assumed even and $k \equiv 1, 2, 9, 10 \pmod{12}$, we get $\tau(m) \equiv 2,10 \pmod{12}$. This implies that $\frac{\tau(m)}{2} \ge \frac{k}{2} > 1$ is not divisible by $2, 3$ or $4$, so that $e_i \ge 4$ for all but at most one $i$. We therefore deduce that $m$ is divisible by either $a_1 = 3^4 \cdot 5, a_2 = 3^4 \cdot 7, a_3 = 3\cdot 5^4$ or $a_4 = 3 \cdot 7^4$. Let us deal with them in order.

    The integer $a_1$ is divisible by $9$ and $15$, $a_2$ is divisible by $21$ and $27$, and $a_3$ is divisible by $75$ and $125$, none of which are allowed. Finally, if $a_4$ divides $m$, then $m$ needs a further odd divisor $d$ with $8 < d < 16$. If $d = 9$, then $m$ is divisible by $49$ and $63$. If $d \in \{11, 13\}$, then $m$ is divisible by $3d$ and $49$. And if $d = 15$, then $m$ is divisible by $5$ and $7$. Since every case leads to two divisors in between two consecutive powers of two, this finishes the proof.
\end{proof}

On the other hand, the lower density of $k$ such that $2^k$ is separable, is also positive. Before we formally state and prove this, let us first recall the following two results from~\cite[Prop.~6.8]{Dusart10} and~\cite{Dudek} on prime gaps. 

\begin{lem} \label{lem:bhp}
    For all $N \ge 2^{2^5}$, there exists a prime in the interval $\left(N, N\left(1 + \frac{1}{25 \log^2 N}\right)\right]$. Moreover, for all $N \ge 2^{2^{50}}$, there exists a prime in the interval $(N, N + 3N^{2/3}].$
\end{lem}

Secondly, we will also need the following result from~\cite[Thm.~3]{Tenenbaum84}.

\begin{lem} \label{lem:tenenbaum}
    There exists an absolute constant $c$ such that for all $x, y, z$ with $2 \le y \le z \le x$ there are at most $\frac{cx \log y}{\log z}$ positive integers $n \le x$ which do not have a divisor in the interval $[y, z]$. 
\end{lem}

Now, define $t := \max(5 + 3 \log c, 50)$ with $c$ as in~\cref{lem:tenenbaum}. Furthermore define $C := \log 2 \cdot 2^{t-2}$, and let $S$ be the set of positive integers $n$ with the property that if $1 = d_1 < \cdots < d_{\tau(n)} = n$ are the divisors of $n$, then $d_j \le e^{\max(C, d_{j-1})}$ for all $j$ with $2 \le j \le \tau(n)$. Finally, for a given integer $x$, let $S(x)$ be the set of integers $n \le x$ with $n \in S$. We then claim that $S$ contains more than half of all positive integers.

\begin{lem}\label{lem:rela_consec_divisors}
    For every $x \in \mathbb{N}$ we have $\abs{S(x)} > \frac{x}{2}$.
\end{lem}

\begin{proof}
     We may, by the definition of $S(x)$, without loss of generality assume $x \ge e^C$. So with $l$ defined as $\left \lfloor \log_2 \log_C \log x \right \rfloor$, we then get $l \ge 0$. We furthermore define $$y_i := \exp(C^{2^{i-1}}) \qquad \text{and} \qquad z_i := \exp(C^{2^{i}})$$ for all $i$ with $0 \le i \le l$. By \cref{lem:tenenbaum}, the number of positive integers $n \le x$ for which an $i$ exists such that $n$ does not have a divisor in the interval $[y_i, z_i]$, is at most $\frac{cx \log y_{0}}{\log z_{0}} + \cdots + \frac{cx \log y_{l}}{\log z_{l}}$. Applying $t = \max(5 + 3 \log c, 50)$, we will bound this sum by using the lower bound $$C = \log 2 \cdot 2^{t-2} \ge \max\left(8\log 2 \cdot 2^{3 \log c}, \log 2 \cdot 2^{48}\right) > \max\left(5c^2, 100\right)$$ as follows:
\begin{align*}
    \sum_{i=0}^l \frac{cx \log y_i}{\log z_i} &= \sum_{i=0}^{l} \frac{cx}{C^{2^{i-1}}} \\
    &< cx\left(\frac{1}{\sqrt{C}} + \frac{1}{C} + \sum_{i=1}^{\infty} \frac{1}{C^{2i}}\right) \\
    &= cx\left(\frac{1}{\sqrt{C}} + \frac{1}{C} + \frac{1}{C^2 - 1}\right) \\
    &< \frac{10cx}{9\sqrt{C}} \\
    &< \frac{x}{2}.
\end{align*}

    Here, the final two inequalities use $C > 100$ and $C > 5c^2$ respectively. In other words, the number of $n \le x$ such that $n$ has a divisor in the interval $[y_i, z_i]$ for every $i$, is more than $\frac{x}{2}$. 

    Now, if $n \in \mathbb{N}$ is such that $n$ indeed has a divisor in all intervals $[y_i, z_i]$, then we in particular get that for consecutive divisors $d_{j-1}, d_j$ of $n$, either $d_j \le z_{0} = e^{C}$, or there exists an $i \ge 0$ such that $$y_i \le d_{j-1} < d_j \le z_{i+1}.$$ In the second case we have
\begin{align*}
    \log \log d_j &\le \log \log z_{i+1} \\
    &= 2^{i+1} \log C \\
    &< C^{2^{i-1}} \\
    &= \log y_i \\
    &\le \log d_{j-1}.
\end{align*}

    Or, in other words, $d_j < e^{d_{j-1}}$. In either case we have $d_j \le e^{\max(C, d_{j-1})}$, so we conclude $n \in S(x).$
\end{proof}

    The above lemma will help us show that the lower density of $k$ such that $2^k$ is separable, is positive. 

\begin{thm}\label{thm:posdens_2^kseparable}
    If $k \in 2^t S = \{k : k = 2^t k' \text{ for some } k' \in S\}$, then $2^k$ is separable. In particular, the lower density of the set of positive integers $k$ for which $2^k$ is separable, is at least $\frac{1}{2^{t+1}}$.
\end{thm}

\begin{proof}
    For an integer $k \in 2^t S$, let $\prod_{i = 1}^r (e_i+1)$ be the prime factorisation of $k$. That is, $e_i+1$ is prime for all $i$, $e_i \le e_{i+1}$ for all $i$, $e_i = 1$ for all $i \le t$, and $k = \prod_i (e_i+1)$. For all $i$ with $4 \le i \le r$, define $n_i$ to be $2^{(e_1+1) \cdots (e_{i-1} + 1)}$, and let $p_i$ be the smallest prime larger than $n_i$. Then we will prove that $m = 231 \prod_{i=4}^{r} p_i^{e_i}$ interlocks with $n = 2^k$.

    More precisely, let $1 = d_1 < \cdots < d_8 = 231$ be the divisors of $231 = 3 \cdot 7 \cdot 11$, let $2^d > 1$ be a divisor of $n$, and write $$d = c_0 + \sum_{i = 4}^r c_i(e_1+1) \cdots (e_{i-1} + 1)$$ with $0 \le c_0 \le 7$ and $0 \le c_i \le e_i$ for all $i \ge 4$.\footnote{It follows by induction that it is indeed possible to write $d$ in such a way.} With the divisor $$d' := d_{c_0 + 1} \prod_{i = 4}^r p_i^{c_i}$$ of $m$, we will then show $2^d < d' < 2^{d+1}$. Since $$\tau(m) = 8 \prod_{i=4}^r (e_i+1) = \prod_{i=1}^r (e_i+1) = k = \tau(n) - 1,$$ we see that such a $d'$ would be unique and the proof would be finished. 

    Because $2^{c_0} \le d_{c_0 + 1}$ for all $c_0$ with $0 \le c_0 \le 7$, and $n_i^{c_i} < p_i^{c_i}$ by definition of $p_i$, the inequality $2^d < d'$ quickly follows. As for the inequality $d' < 2^{d+1}$, we need some upper bounds.
		
\begin{claim}\label{claim:est_ej}
    For all $i$ with $t+1 \le i \le r$, we have $e_i \le n_i^{1/4}.$
\end{claim}
		
\begin{claimproof}
    Let $i$ be fixed and recall that $C = \log 2 \cdot 2^{t-2}$. We then first of all get the estimate $$e^C = 2^{2^{t-2}} = n_{t+1}^{1/4} \le n_i^{1/4}.$$ Secondly, with $D := (e_{t+1} + 1) \cdots (e_{i-1} + 1)$, we have $$e^D < 2^{2D} = \left(2^{8D}\right)^{1/4} < n_i^{1/4},$$ as $2^t > 8$. Now, either $$e_i \le D < e^D \le n_i^{1/4},$$ or $e_i > D$. In the first case we are done, while in the latter case we note that both $D$ and $e_i+1$ divide $2^{-t}k$. As $D$ is defined as the product of all prime divisors of $2^{-t}k$ which are smaller than $e_i + 1$, we then see that $D$ and $e_i+1$ must actually be consecutive divisors of $2^{-t}k$. And we then still get $$e_i+1 \le e^{\max(C, D)} \le n_i^{1/4},$$ since $2^{-t}k \in S$ by assumption.	
\end{claimproof}

    By combining~\cref{claim:est_ej} with~\cref{lem:bhp} we can show an upper bound on $p_i$.

\begin{claim}\label{claim:pi_lower_bound}
    For all $i$ with $4 \le i \le r$, we have $\left(\frac{p_i}{n_i}\right)^{e_i} < e^{1/4^{i}}.$
\end{claim}

\begin{claimproof}
    Recall that $e_i = 1$ for all $i \le t$. Now, with $p_4=257$ and $p_5=65537$ one can check that the lemma holds for $i \in \{4, 5\}$. For $i \ge 6$ we apply~\cref{lem:bhp}, from which we in particular get that there is a prime in the interval $\left(n_i, n_i\left(1 + \frac{1}{25 \log^2 n_i}\right)\right].$ Since $\log(n_i) = \frac{\log 2}{2} \cdot 2^{i}$ if $i \le t$, while $\frac{25 \log^2 2}{4} > 1$ and $1 + \frac{1}{4^i} < e^{1/4^{i}}$, the claim follows for $6 \le i \le t$. As for $t+1 \le i \le r$, we know by $t \ge 50$ that there is actually a prime in the interval $\left(n_i, n_i + 3n_i^{2/3}\right]$. This gives
\begin{align*}
    \left(\frac{p_i}{n_i} \right)^{e_i} &\le \left(1 + 3n_i^{-1/3}\right)^{e_i} \\
    &< e^{3e_i n_i^{-1/3}} \\
    &\le e^{3n_i^{-1/12}} \\
    &< e^{1/4^i}.
\end{align*}

    The second to last inequality follows from~\cref{claim:est_ej}. As for the final inequality, note that we have $$i+1 < \frac{1}{12} \cdot 2^{i-2}$$ by $i > t > 9.$ We therefore see $$3 \cdot 4^i < 4^{i+1} < 2^{2^{i-1}/12} \le n_i^{1/12},$$ showing our claim.
\end{claimproof}

    Now, one can check that $\frac{d_{c_0+1}}{2^{c_0+1}} < \frac{10}{11}$ holds for all $c_0 \le 7$. We moreover have the upper bound
\begin{align*}
    \frac{p_{4}^{c_4} \cdots p_{r}^{c_r}}{n_{4}^{c_4} \cdots n_{r}^{c_r}} &\le \prod_{i = 4}^{r} \left(\frac{p_i}{n_i}\right)^{e_i} \\
    &< \prod_{i = 4}^{\infty} e^{1/4^i} \\
    &= e^{1/192} \\
    &< \frac{11}{10}.
\end{align*}

    We therefore conclude $$\frac{d'}{2^{d+1}} = \frac{d_{c_0+1}p_{4}^{c_4} \cdots p_{r}^{c_r}}{2^{c_0+1} n_{4}^{c_4} \cdots n_{r}^{c_r}} < 1,$$ as desired.
\end{proof}

For our final result, define $P_k$ to be the product of the first $k$ primes. If $k > 1$ is odd, then the product $mn$ of an interlocking pair cannot be equal to $P_k$ by~\cref{lem:tau}, as $$\abs{\tau(m) - \tau(n)} \ge 2^{\frac{k-1}{2}}  > 1$$ in that case. However, even if $k$ is even, Erd\H{o}s and Hall conjectured that the product $mn$ of an interlocking pair cannot be equal to $P_k$, if $k$ is large enough. This turns out to be true.

\begin{thm}\label{thm:onlyfewproducts}
    If $(m,n)$ is an interlocking pair with $mn = P_k$, then $k\le 8.$
\end{thm}

\begin{proof}
    Let $k$ be larger than $8$ and assume without loss of generality that $m$ is divisible by $2$. This implies $3 \mid n$ and $5 \mid m$, which (due to the divisor $10$ of $m$) in turn implies that $7$ and $11$ must divide $n$. Continuing this process gives $p \mid m$ for $p \in \{2, 5, 13, 19, 23\}$ and $q \mid n$ for $q \in \{3, 7, 11, 17\}$. However, we now reach a contradiction as $23$ and $26$ are both divisors of $m$, while $n$ cannot be a multiple of $24$ or $25$.
\end{proof}

As Erd\H{o}s and Hall note, for even $k\le 8,$ it is possible for an interlocking pair $(m,n)$ to have $mn = P_k$, by taking $m$ and $n$ to be the product of the first $\frac k2$ primes in $\{2,5,13,19\}$ and $\{3,7,11,17\}$ respectively.



\begin{thebibliography}{1}

\bibitem{LeanFile}
W. van Doorn, 
\newblock {{Erd{\H{o}}sHall.lean}}, GitHub repository. 
\url{https://github.com/Woett/Lean-files/blob/main/Erd%C5%91sHall.lean}

\bibitem{Dudek}
A.~Dudek.
\newblock An explicit result for primes between cubes.
\newblock {\em Funct. Approx. Comment. Math.}, 55(2):177--197, 2014.

\bibitem{Dusart10}
P.~Dusart.
\newblock Estimates of {Some} {Functions} {Over} {Primes} without {R}.{H}.
\newblock Preprint, {arXiv}:1002.0442 [math.{NT}] (2010), 2010.

\bibitem{EH78}
P.~Erd{\H{o}}s and R.~R. Hall.
\newblock On some unconventional problems on the divisors of integers.
\newblock {\em J. Aust. Math. Soc., Ser. A}, 25:479--485, 1978.

\bibitem{Tenenbaum84}
G.~Tenenbaum.
\newblock On the probability that an integer has a divisor in a given interval.
\newblock {\em Compos. Math.}, 51:243--263, 1984.

\end{thebibliography}
\end{document}